\documentclass[english,4paper,10pt]{article}
\usepackage[utf8]{inputenc}
\usepackage[T1]{fontenc}
\usepackage[english]{babel}
\usepackage[autostyle=true]{csquotes}
\usepackage{latexsym, amssymb, amsmath}
\usepackage{nameref,hyperref,cleveref}
\usepackage{amsthm}
\makeatletter
\renewcommand*\env@matrix[1][*\c@MaxMatrixCols c]{%
	\hskip -\arraycolsep
	\let\@ifnextchar\new@ifnextchar
	\array{#1}}
\makeatother
\usepackage{lipsum}
\usepackage{url}
\usepackage{float}
\usepackage{caption}
\usepackage{subcaption}
\usepackage{tikz-cd}
\usepackage{color}  
\usepackage{hyperref}
\usepackage[symbol]{footmisc}
\hypersetup{
	colorlinks=true,
	linkcolor=black,
	filecolor=magenta,      
	urlcolor=cyan,
	linktoc=all,    
}
\usepackage[ruled,vlined]{algorithm2e}
\usepackage{graphicx}
\graphicspath{ {./images/} }
\usepackage{amscd} 

\usepackage{tikz-cd}
\usepackage[all,cmtip]{xy}

\newcommand{\R}{\mathbb{R}}

\newcommand{\F}{\mathbb{F}}

\newcommand{\mg}{\mathfrak{g}}
\newcommand{\mh}{\mathfrak{h}}
\newcommand{\mpp}{\mathfrak{p}}
\newcommand{\bigzero}{\mbox{\normalfont\Large\bfseries 0}}

\newtheorem{thm}{Theorem}[section]
\newtheorem{lemma}[thm]{Lemma}

\newtheorem{definition}[thm]{Definition}
\newtheorem{rem}[thm]{Remark}
\newtheorem{prop}[thm]{Proposition}
\newtheorem{ex}[thm]{Example}

\makeatletter
\def\blfootnote{\xdef\@thefnmark{}\@footnotetext}
\makeatother

\date{}

\usepackage[%
backend=bibtex,%
style=numeric-comp,%
sorting=nyt,%
hyperref,%
maxcitenames=4, mincitenames=4,
maxbibnames=99, minbibnames=99 %
]{biblatex}

\AtBeginBibliography{}

\renewbibmacro{in:}{%
	\ifentrytype{article}
	{}
	{\bibstring{in}%
		\printunit{\intitlepunct}}}

\DeclareUnicodeCharacter{2212}{-}
\begin{document}
	\sloppy
	
	\title{Isotopisms of nilpotent \\ Leibniz algebras and Lie racks \blfootnote{\textit{\textup{2020} Mathematics Subject Classification}: 17A32, 17A36, 17B30, 20M99, 22A30.}} \maketitle \noindent
	{{Gianmarco La Rosa}\footnote[2]{The first two authors are supported by University of Palermo, by the “National Group for Algebraic and Geometric Structures, and their Applications” (GNSAGA – INdAM), by the National Recovery and Resilience Plan (NRRP), Mission 4, Component 2, Investment 1.1, Call for tender No.\ 1409 published on 14/09/2022 by the Italian Ministry of University and Research (MUR), funded by the European Union -- NextGenerationEU -- Project Title \textbf{Quantum Models for Logic, Computation and Natural Processes} (\textbf{QM4NP}) -- CUP \textbf{B53D23030160001} -- Grant Assignment Decree No.\ 1371 adopted on 01/09/2023 by the Italian Ministry of Ministry of University and Research (MUR), and by the \textbf{Sustainability Decision Framework} (\textbf{SDF}) Research Project --  CUP \textbf{B79J23000540005} -- Grant Assignment Decree No.\ 5486 adopted on 04/08/2023. 
	The third author is supported by the NKFIH-OTKA Grants SNN 132625, and by the Program of Excellence TKP2021-NVA-02 at the Budapest University of Technology and Economics.}, {Manuel Mancini}\footnotemark[2]
\\  \footnotesize{Dipartimento di Matematica e Informatica, Università degli Studi di Palermo}\\
    \footnotesize{Via Archirafi 34, 90123 Palermo, Italy}\\
		\footnotesize{gianmarco.larosa@unipa.it}, ORCID: 0000-0003-1047-5993 \\
	    \footnotesize{manuel.mancini@unipa.it}, ORCID: 0000-0003-2142-6193}
	
	\bigskip\noindent
	{{Gábor P.\ Nagy\footnotemark[2]} \\	
		\footnotesize{Department of Algebra and Geometry, Budapest University of Technology and Economics\\M\H{u}egyetem rkp 3, H-1111 Budapest, Hungary\\ \\
		Bolyai Institute, University of Szeged\\Aradi V\'ertan\'uk Tere 1, H-6720 Szeged, Hungary}\\
		\footnotesize{nagyg@math.u-szeged.hu}, ORCID: 0000-0002-9558-4197}
	
	\begin{abstract}
		In this paper we study the isotopism classes of two-step nilpotent algebras. We show that every nilpotent Leibniz algebra $\mathfrak{g}$ with $\dim[\mathfrak{g},\mathfrak{g}]=1$ is isotopic to the Heisenberg Lie algebra or to the Heisenberg algebra $\mathfrak{l}_{2n+1}^{J_1}$, where $J_1$ is the $n \times n$ Jordan block of eigenvalue 1. We also prove that two such algebras are isotopic if and only if the Lie racks integrating them are isotopic. This gives the classification of Lie racks whose tangent space at the unit element is a nilpotent Leibniz algebra with one-dimensional commutator ideal. Eventually, we introduce new isotopism invariants for Leibniz algebras and Lie racks.
	\end{abstract}
	
\bigskip
	
\section*{Introduction}
	
Leibniz algebras were first introduced by J.-L.\ Loday in \cite{loday1993version} as a non-antisymmetric version of Lie algebras. Earlier, such algebraic structures were considered by A.\ Blokh, who called them D-algebras \cite{blokhLie} for their strict connection with the derivations. Leibniz algebras play a significant role in different areas of mathematics and physics.

Many results of Lie algebras were also established in the frame of Leibniz algebras. One of them is the \emph{Levi decomposition} (see \cite{ayupov2019leibniz}), which states that every finite-dimensional Leibniz algebra $\mathfrak{g}$ over a field $\F$, with $\operatorname{char}(\F)=0$, is the semidirect product of a solvable ideal (the \emph{radical} of $\mg$) and a semisimple subalgebra, which is a Lie algebra. This makes clear the importance of the problem of Lie and Leibniz algebras classification, which has been dealt with since the 20th century (see \cite{bartolone2011nilpotent}, \cite{BarDiFal18}, \cite{3-dimApuyov}, \cite{3-dimLadra}, \cite{3-dimF}, \cite{4dimLeib} and \cite{6dimLeib}, just for giving some examples). However in general, given two Leibniz algebras $\mathfrak{g}$ and $\mathfrak{h}$, it is hard to check if $\mathfrak{g}$ and $\mathfrak{h}$ are isomorphic or not. As we will see, in some cases it it is easier to verify if there exists an \emph{isotopism} between them, i.e.\ if there is a triple of linear isomorphisms $(f,g,h)\colon \mathfrak{g} \leftrightarrows \mathfrak{h}$ such that $[f(x),g(y)]_{\mathfrak{h}}=h([x,y]_{\mathfrak{g}})$, for every $x,y \in \mathfrak{g}$.

In \cite{LM2022} two-step nilpotent algebras over a field $\mathbb{F}$, with $\operatorname{char}(\F) \neq 2$, were classified and integrated into Lie racks. However, the problem of finding the isomorphism classes of these algebras is still open, also in the case that $\F$ is the field of complex numbers.

In this paper our goal is to use the notion of \emph{isotopism}, which was introduced by A.\ A.\ Albert in 1942 (see \cite{Albert}, \cite{Albert2} and \cite{IsotopismStory}) in order to classify non-associative algebras, and then used in \cite{IsotopismLie2} and \cite{IsotopismLie1} for classifying \emph{filiform Lie algebras} over finite fields, to find  a complete classification of nilpotent Leibniz algebras with one-dimensional commutator ideal. We prove that two such algebras are isotopic if and only if the Lie racks integrating them are isotopic. By using this result, we are able to find the isotopism classes of Lie racks whose tangent space at the unit element is a nilpotent Leibniz algebra with one-dimensional commutator ideal. 

In Sections 1 and 2 we resume the main results about the classification and the integration of two-step nilpotent algebras  \cite{LM2022, LaRosaMancini2}. We recall that, up to isomorphism, there are only three classes of indecomposable nilpotent Leibniz algebras with one-dimensional commutator ideal, namely the \emph{Heisenberg algebras} $\mathfrak{l}_{2n+1}^A$, parametrized by their dimension $2n+1$ and a $n \times n$ matrix $A$ in canonical form, the \emph{Kronecker algebras} $\mathfrak{k}_n$ and the \emph{Dieudonné algebras} $\mathfrak{d}_n$, both parametrized by their dimension only. We also recall the construction of the Lie rack integrating a two-step nilpotent algebra.

In Section 3 we give the definitions of isotopism of Leibniz algebras and Lie racks. We recall that every isotopism is isomorphic to a so-called \emph{principal isotopism} and we give the explicit construction of an isotopism between the Heisenberg Lie algebra $\mh_3$ and a nilpotent Leibniz non-Lie algebra.

In Section 4 we introduce new isotopism invariants for Leibniz algebras and Lie racks and we prove that every indecomposable nilpotent Leibniz algebra $\mg$ with $\dim_\F [\mg,\mg]=1$ is either isomorphic to the Heisenberg algebra $\mathfrak{l}_{2n+1}^{J_1}$, where $J_1$ is the $n \times n$ Jordan block of eigenvalue $1$, or it is isotopic to the Heisenberg Lie algebra $\mathfrak{h}_{2n+1}$. Finally we show that, if $\mg$ and $\mh$ are real Leibniz algebras with $\dim_\R[\mg,\mg]=\dim_\R[\mh,\mh]=1$ and $X,Y$ are the Lie racks integrating them, then $\mg$ is isotopic to $\mh$ if and only if $X$ is isotopic to $Y$. This implies that a Lie rack $X$ integrating a Leibniz algebra $\mg$, with $\dim_\R [\mg,\mg]=1$, is either isomorphic to the Heisenberg rack $R_{2n+1}^{J_1}$, or it is isotopic to the conjugation $\operatorname{Conj}(H_{2n+1})$ of the Heisenberg Lie group.
	
\section{Preliminaries}
	
We assume that $\F$ is a field with $\operatorname{char}(\F)\neq2$. For the general theory we refer to \cite{ayupov2019leibniz} and \cite{loday1993version}.
	
\begin{definition}
A \emph{left Leibniz algebra} over $\mathbb{F}$ is a vector space $\mg$ over $\mathbb{F}$ endowed with a bilinear map (called $commutator$ or $bracket$) $\left[-,-\right] \colon \mg \times \mg \rightarrow \mg$ which satisfies the \emph{left Leibniz identity}
$$
\left[x,\left[y,z\right]\right]=\left[\left[x,y\right],z\right]+\left[y,\left[x,z\right]\right],\quad \forall x,y,z\in \mg.
$$
		
\end{definition}
	In the same way one can define a right Leibniz algebra, using the right Leibniz identity
	$$
	\left[\left[x,y\right],z\right]=\left[[x,z],y\right]+\left[x,\left[y,z\right]\right],\quad \forall x,y,z\in \mg.
	$$
	A Leibniz algebra that is both left and right is called \emph{symmetric Leibniz algebra}. From now on we assume that $\dim_\F \mg<\infty$.
	
	Every Lie algebra is a Leibniz algebra and every Leibniz algebra with skew-symmetric commutator is a Lie algebra.  The \emph{Leibniz kernel} of $\mg$
	$$
	\operatorname{Leib}(\mg)=\langle \left[x,x\right] |\,x\in \mg \rangle
	$$
is the smallest two-sided ideal of $\mg$ such that $\mg /\operatorname{Leib}(\mg)$ is a Lie algebra. This defines an adjunction \cite{mac2013categories} between the category $\textbf{Lie}$ of Lie algebras over $\F$ and the category $\textbf{Leib}$ of Leibniz algebras over $\F$. The left adjoint of the full inclusion $i \colon \textbf{Lie}\rightarrow \textbf{Leib}$ is $\pi \colon \textbf{Leib} \rightarrow \textbf{Lie}$, defined by $\pi(\mg)=\mg/\operatorname{Leib}(\mg)$.

	For a Leibniz algebra $\mg$, the left center and the right center are respectively
	$$
	\operatorname{Z}_l(\mg)=\left\{x\in \mg \,|\,\left[x,\mg\right]=0\right\},\,\,\, \operatorname{Z}_r(\mg)=\left\{x\in \mg \,|\,\left[\mg,x\right]=0\right\}
	$$
	and they coincide when $\mg$ is a Lie algebra. The \emph{center} of $\mg$ is $\operatorname{Z}(\mg)=\operatorname{Z}_l(\mg)\cap \operatorname{Z}_r(\mg)$. We recall that in general $\operatorname{Z}_l(\mg)$ is an ideal of the left Leibniz algebra $\mg$, meanwhile the right center is not even a subalgebra. Moreover $\operatorname{Leib}(\mg) \subseteq \operatorname{Z}_l(\mg)$, since
	$$
	[[x,x],y]=[x,[x,y]]-[x,[x,y]]=0, \quad \forall x,y \in \mg.
	$$
	
	We recall now the definitions of solvable and nilpotent Leibniz algebra.
	
	\begin{definition}
		Let $\mg$ be a left Leibniz algebra over $\F$ and let
		$$
		\mg^{0}=\mg,\quad \mg^{k+1}=[\mg^k,\mg^{k}],\quad\forall k\geq0
		$$ 
		be the \emph{derived series of $\mg$}. $\mg$ is \emph{$n$-step solvable} if $\mg^{n-1}\neq0$ and $\mg^{n}=0$.
	\end{definition}
	
	\begin{definition}
		Let $\mg$ be a left Leibniz algebra over $\F$ and let
		$$
		\mg^{(0)}=\mg,\quad \mg^{(k+1)}=[\mg,\mg^{(k)}],\quad\forall k\geq0
		$$ 
		be the \emph{lower central series of $\mg$}. $\mg$ is \emph{$n$-step nilpotent} if $\mg^{(n-1)}\neq0$ and $\mg^{(n)}=0$.
	\end{definition}
	
	When $n=2$, two-step nilpotent Leibniz algebras lie in different varieties of non-associative algebras, such as associative, alternative and Zinbiel algebras. For this reason we will refer at them as \emph{two-step nilpotent algebras} and we have the following.
	
	\begin{prop}\cite{LM2022}
		Let $\mg$ be a $n$-step nilpotent Leibniz algebra.
		\begin{itemize}
			\item[(i)] If $n=2$, then $\mg^{(1)}=[\mg,\mg]\subseteq \operatorname{Z}(\mg)$ and $\mg$ is a symmetric Leibniz algebra;
			\item[(ii)] If $\dim_\F [\mg,\mg]=1$, then $n=2$.
		\end{itemize}
	\end{prop}

Nilpotent Leibniz algebras with one-dimensional commutator ideal were classified in \cite{LM2022}. It was proved that, up to isomorphism, there are only the following three classes of indecomposable nilpotent Leibniz algebras with one-dimensional commutator ideal.

\begin{definition}\cite{LM2022}
	Let $f(x)\in\F\left[x\right]$ be a monic irreducible polynomial. Let $k\in\mathbb{N}$ and let $A=(a_{ij})_{i,j}$ be the companion matrix of $f(x)^k$. The \emph{Heisenberg} Leibniz algebra $\mathfrak{l}_{2n+1}^A$ is the $(2n+1)$-dimensional Leibniz algebra with basis $\left\{e_1,\ldots,e_n,f_1,\ldots,f_n,z\right\}$ and the Leibniz brackets are given by
	$$
	[e_i,f_j]=(\delta_{ij}+a_{ij})z, \; [f_j,e_i]=(-\delta_{ij}+a_{ij})z, \; \; \forall i,j=1,\ldots,n.
	$$
\end{definition}
	
	If $A$ is the zero matrix, then we obtain the Heisenberg algebra $\mathfrak{h}_{2n+1}$, i.e.\ the $(2n+1)-$dimensional Lie algebra with $[e_i,f_j]=\delta_{ij}h$, for any $i,j=1,\ldots,n$, where $\delta_{ij}$ is the \emph{Kronecker delta}.

\begin{definition}\cite{LM2022}
	Let $n \in \mathbb{N}$. The \emph{Kronecker} Leibniz algebra $\mathfrak{k}_{n}$ is the $(2n+1)$-dimensional Leibniz algebra with basis $\left\{e_1,\ldots,e_n,f_1,\ldots,f_n,z\right\}$ and the Leibniz brackets are given by
	\begin{align*}
		&[e_i,f_i] = [f_i,e_i] = z, \; \; \forall i=1,\ldots,n \\
		& [e_i,f_{i-1}] = z, [f_{i-1},e_i] = −z, \; \; \forall i= 2,\ldots,n.
	\end{align*}
\end{definition}

\begin{definition}\cite{LM2022}
	Let $n \in \mathbb{N}$. The \emph{Dieudonné} Leibniz algebra $\mathfrak{d}_n$ is the $(2n+2)$-dimensional Leibniz algebra with basis $\left\{e_1,\ldots,e_{2n+1},z\right\}$ and the Leibniz brackets are given by
	\begin{align*}
		&\left[e_1,e_{n+2}\right]=z,\\
		&\left[e_i,e_{n+i}\right]=\left[e_i,e_{n+i+1}\right]=z, \; \; \forall i=2,\ldots,n,\\
		&\left[e_{n+1},e_{2n+1}\right]=z,\\
		&\left[e_i,e_{i-n}\right]=z, \; \; \left[e_i,e_{i-n-1}\right]=-z, \; \; \forall i=n+2,\ldots,2n+1.
	\end{align*}
	
\end{definition}
	
	For every $n\in\mathbb{N}$, the Kronecker algebra $\mathfrak{k}_{n}$ and the Dieudonné algebra $\mathfrak{d}_{n}$ are not Lie algebras and they are unique up to isomorphisms.
	
	Now we recall the description of the Heisenberg algebras in the case the field $\F$ is $\mathbb{C}$ or $\R$.
	
	\subsection{Complex Heisenberg algebras}\label{HeisC}
	Let $n\in\mathbb{N}$ and let $f(x)=x-a\in\mathbb{C}[x]$. Then the companion matrix $A$ of $f(x)^n$ is similar to the $n\times n$ Jordan block $J_a$ of eigenvalue $a$. Thus $\mathfrak{l}_{2n+1}^A\cong\mathfrak{l}_{2n+1}^{J_a}$ and the non-trivial commutators are
	\begin{align*}
			&\left[e_i,f_i\right]=(1+a)h, \quad \left[f_i,e_i\right]=(-1+a)h, \quad \forall i=1,\ldots,n,\\
		&\left[e_i,f_{i-1}\right]=\left[f_{i-1},e_i\right]=h, \quad \forall i=2,\ldots,n,
	\end{align*}
	where $\left\{e_1,\ldots,e_n,f_1,\ldots,f_n,h\right\}$ is a basis of $\mathfrak{l}_{2n+1}^{J_a}$. One can check that
	$$
	\operatorname{Z}_l(\mathfrak{l}_{2n+1}^{J_a}))=\operatorname{Z}_r(\mathfrak{l}_{2n+1}^{J_a}))=[\mathfrak{l}_{2n+1}^{J_a},\mathfrak{l}_{2n+1}^{J_a}]=\F h, \quad \forall a \neq \pm 1
	$$
	$$
	\operatorname{Z}_l(\mathfrak{l}_{2n+1}^{J_1})=\operatorname{Z}_r(\mathfrak{l}_{2n+1}^{J_{-1}})=\langle f_n,h \rangle,
	$$
	$$
	\operatorname{Z}_r(\mathfrak{l}_{2n+1}^{J_1})=\operatorname{Z}_l(\mathfrak{l}_{2n+1}^{J_{-1}})=\langle e_1,h \rangle.
	$$
	and we have the following.
	
	\begin{prop}\label{propHeis} \cite{LM2022} 
		Let $a \in \mathbb{C}$. The Heisenberg algebras $\mathfrak{l}_{2n+1}^{J_a}$ and $\mathfrak{l}_{2n+1}^{J_{-a}}$ are isomorphic.
	\end{prop}
	
	We observe that the problem of finding the isomorphism classes of the complex Heisenberg algebras is still open, but when $n=1$ the converse result of Proposition \ref{propHeis} is also true.
	
	\begin{prop}\cite{LM2022}
		Let $a,a'\in\mathbb{C}$. Then $\mathfrak{l}_3^a$ and $\mathfrak{l}_3^{a'}$ are isomorphic if and only if $a'=\pm a$.
	\end{prop}

In Theorem \ref{thmiso} we will show that, for any $a \neq \pm 1$, there exists an \emph{isotopism} between $\mathfrak{l}_{2n+1}^{J_a}$ and the Heisenberg Lie algebra $\mathfrak{h}_{2n+1}$.
	
	\subsection{Real Heisenberg algebras}\label{secHeisReal}
	 Let $f(x)\in\R[x]$ be an irreducible monic polynomial. If $f(x)=x-a$, then $\mathfrak{l}_{2n+1}^A\cong\mathfrak{l}_{2n+1}^{J_a}$, as in the complex case. So we suppose that $f(x)=x^2+bx+c$, with $b^2-4c<0$.

	Let $z=\alpha + i \beta \in \mathbb{C}$ be a root of $f(x)$. Then $f(x)=(x-z)(x-\bar{z})$ and the companion matrix $A$ of $f(x)^n$ in similar to the $2n \times 2n$ real block matrix
	$$
	J_R=\begin{pmatrix}
		R&0&\cdots&0\\I_2&R&\cdots&0\\\vdots&\ddots&\ddots&\vdots\\0&\cdots&I_2&R
	\end{pmatrix}, \quad \text{where} \; \; R=\begin{pmatrix}
	\alpha & \beta \\
	-\beta & \alpha
\end{pmatrix}
	$$
	is the realification of the complex number $z$. Thus $\mathfrak{l}_{4n+1}^{A} \cong \mathfrak{l}_{4n+1}^{J_R}$ and we have the following, when $n=1$.
	\begin{prop}\cite{LM2022}
		Let $f(x),g(x) \in \mathbb{R}[x]$ be irreducible monic polynomials of degree two. Let $z,z' \in \mathbb{C}$ be roots of $f(x)$ and $g(x)$ respectively and let $R,R'$ be the realifications of $z$ and $z'$. Then $\mathfrak{l}_5^{R} \cong \mathfrak{l}_5^{R'}$ if and only if $R'=\pm R$.
	\end{prop}

\section{The \emph{coquecigrue problem} for nilpotent Leibniz algebras}

In this section, the underlying field of any vector space is $\F=\R$.

\begin{definition}
		A \emph{(left) rack} is a set $X$ with a binary operation $\rhd\colon X\times X\rightarrow X$ which is \emph{(left) autodistributive}, 
		$$x\rhd\left(y\rhd z\right)=\left(x\rhd y\right)\rhd\left(x\rhd z\right), \quad \forall x,y,z \in X$$
		 and such that $x\rhd - \colon X\rightarrow X$ is a bijection for every $x\in X$. A rack is \emph{pointed} if there exists an element $1\in X$, called the \emph{unit}, such that $1\rhd x=x$ and $x\rhd 1=1$, for every $x\in X$. A rack $(X,\rhd)$ is a \emph{quandle} if $x\rhd x=x$, for every $x\in X$ (i.e.\ $\rhd$ is idempotent).
\end{definition}

Here we work with pointed racks.

\begin{definition}
	Let $(X,\rhd,1)$ and $(Y,\rhd_Y,1_Y)$ be pointed racks. A \emph{pointed rack homomorphism} is a map $f \colon X\rightarrow Y$such that $f(1)=1_Y$ and $f(x\rhd y)=f(x)\rhd_Y f(y)$, for every $x,y\in X$.
\end{definition}

\begin{definition}\cite{StabRack}
	A non-empty subset $Z$ of a pointed rack $(X,\rhd,1)$ is a \emph{subrack} of $X$ if $1 \in Z$ and $x \rhd y \in Z$, for every $x,y \in Z$. A subrack $Z$ of $X$ is a \emph{left} (resp.\ \emph{right}) \emph{ideal} if $x \rhd y \in Z$ (resp.\ $y \rhd x \in Z$), for any $x \in X$ and for any $y \in Z$. A subrack $Z$ of $X$ that is both a left and a right ideal is said to be a \emph{two-sided ideal}, or simply an \emph{ideal} of $X$.
\end{definition}

The basic example of a pointed rack is any group $G$ endowed with its conjugation. We denote it by $\operatorname{Conj}(G)$. In this case the unit element is the unit $1_G$ of the group $G$, $\operatorname{Conj}(G)$ is a quandle and its left ideals are precisely the normal subgroups of $G$. More in general, it is defined a functor $\operatorname{Conj}\colon \textbf{Grp}\rightarrow\textbf{PRack},$ where \textbf{Grp} is the category of groups and \textbf{PRack} is the category of pointed racks, which has a left adjoint \cite{mac2013categories} $\operatorname{As}_p\colon\textbf{PRack}\rightarrow\textbf{Grp}$ defined by $\operatorname{As_p}(X)=\operatorname{As}(X)/ \langle [1] \rangle$, where
$$
\operatorname{As}(X)=\operatorname{F}(X)/\overline{\langle \lbrace xyx^{-1} \left(x\rhd y^{-1}\right) \;\vert\; x,y\in X\rbrace \rangle},
$$
$\operatorname{F}(X)$ is the free group generated by $X$ and $\langle [1] \rangle$ is the subgroup of $\operatorname{As}(X)$ generated by the class $[1]$ of the unit element $1$ of the rack $X$. 
\begin{definition}
	A Lie rack is a pointed rack $\left(X,\rhd,1\right)$ such that $X$ is a smooth manifold, $\rhd$ is a smooth map and such that $x\rhd -$ is a diffeomorphism, for every $x \in X$.
\end{definition}

In \cite{kinyon2004leibniz} \mbox{M.\ K.\ Kinyon} showed that the tangent $\operatorname{T}_1 X$ space at the unit element of a \textit{Lie rack} $(X,\rhd,1)$, endowed with the bracket 
$$
[x,y]=\operatorname{ad}_x(y),
$$
where $\operatorname{ad}\colon\operatorname{T}_1X\rightarrow\mathfrak{gl}(\operatorname{T}_1X)$ is the differential of the map $\Phi \colon X \mapsto \operatorname{GL}(\operatorname{T}_1X)$, $\Phi(x)=\operatorname{T}_1(x \rhd -)$, for every $x \in X$, is a Leibniz algebra. 

\begin{prop}{\cite{kinyon2004leibniz}}
	Let $(X,\rhd,1)$ be a Lie rack. Then the above bracket $[x,y]=\operatorname{ad}_x(y)$, defines on $\operatorname{T}_1X$ a Leibniz algebra structure.
\end{prop}

In other words, for every $x,y \in \operatorname{T}_1 X$, we have
\begin{equation*}
	\frac{\partial^2}{\partial s\partial t}\bigg|_{s,t=0} \gamma_1(s) \rhd \gamma_2(t) = [x,y],
\end{equation*}
where $\gamma_1,\gamma_2 \colon [0,1] \rightarrow X$ are smooth paths such that $\gamma_1(0)=\gamma_2(0)=1$, $\gamma_1'(0)=x$ and $\gamma_2'(0)=y$.

The converse problem, that is to find a manifold endowed with a smooth operation such that the tangent space at the distinguished point, endowed with the differential of this operation, gives a Leibniz algebra isomorphic to the given one, is known as the coquecigrue problem. M.\ K.\ Kinyon solved the \emph{coquecigrue problem} for the class of \emph{split Leibniz algebras} \cite{kinyon2004leibniz}. More recently S.~Covez in \cite{covez2013local} showed how to integrate every Leibniz algebra into a local Lie rack. The general coquecigrue problem is still open. In \cite{LM2022} the authors showed that nilpotent real Leibniz algebras admit always a global integration. They also integrated explicitly the indecomposable nilpotent Leibniz algebras with one-dimensional commutator ideal classified in the previous section.

Every left Leibniz algebra $\mathfrak{g}$ can be seen as an abelian extension of a Lie subalgebra $\mg_0\subseteq \mathfrak{gl}(\mg)$ by a $\mg_0-$module $\mathfrak{a}$ \cite{covez2013local}. For instance, one can take $\mathfrak{a}=\operatorname{Z}_l(\mg)$ and $\mg_0=\operatorname{ad}(\mg)\cong\mg/\operatorname{Z}_l(\mg)$, where $\operatorname{ad} \colon \mg \rightarrow \operatorname{gl}(\mg)$, defined by $\operatorname{ad}_x=[x,-]$ for any $x \in \mg$, is the \emph{left adjoint map}. We associate with $\mg$ the short exact sequence
	\begin{equation*}
	\begin{tikzcd}
		0\ar[r]
		&\mathfrak{a} \arrow [r, rightarrowtail]
		&\mg \arrow[r, twoheadrightarrow] &
		\mg_0 \ar[r]
		&0
	\end{tikzcd}
\end{equation*}
and we have a Leibniz algebras $2-$cocycle $\omega\colon \mg_0 \times \mg_0 \rightarrow \mathfrak{a}$ such that $\mg=\mg_0\oplus_\omega \mathfrak{a}$. Moreover the Leibniz algebra structure of $\mg$ is given by
$$
\left[\left(x,a\right),\left(y,b\right)\right]=\left(\left[x,y\right]_{\mg_0}, \rho_x(b)+\omega(x,y)\right),
$$
where $\rho\colon\mg_0\times \mathfrak{a}\rightarrow \mathfrak{a}$ is the action of $\mg_0$ on $\mathfrak{a}$. 
\begin{thm}{\cite{covez2013local}}\label{thmcovez}
	Every Leibniz algebra $\mg=\mg_0\oplus_\omega \mathfrak{a}$ can be integrated into a Lie local rack of the form 
	$$
	G_0\times_f \mathfrak{a}
	$$
	with operation defined by
	$$
	(g,a)\rhd (h,b)=\left(ghg^{-1}, \phi_g(b)+f(g,h)\right)
	$$
	and unit $(1,0)$, where $G_0$ is a Lie group such that $\operatorname{Lie}(G_0)=\mg_0$, $\phi$ is the exponentiation of the action $\rho$, $f\colon G_0\times G_0\rightarrow \mathfrak{a}$ is the Lie local racks $2-cocycle$ defined by 
	$$
	f(g,h)=\int_{\gamma_h}\left(\int_{\gamma_g}\tau^2(\omega)^{eq}\right)^{eq},\quad \forall g,h\in G_0
	$$
	and $\tau^2(\omega)\in \operatorname{ZL}^1(\mg_0,\operatorname{Hom}(\mg_0,\mathfrak{a}))$ is defined by $\tau^2(\omega)(x)(y)=\omega(x,y)$, for every $x,y\in\mg_0$.
\end{thm}

In the case that a Lie rack integrating a Leibniz algebra $\mg$ is a quandle, we have the following.

\begin{thm}\cite{LM2022}\label{thmquandle}
	Let $X$ be a Lie rack integrating a Leibniz algebra $\mg$. $X$ is a quandle if and only if $\mg$ is a Lie algebra. Moreover $X=\operatorname{Conj}(G)$, where $\operatorname{Lie}(G)=\mg$ and $\operatorname{Lie}\colon \mathbf{LieGrp} \rightarrow \mathbf{Lie}$ is the functor which sends every real Lie group $G$ to its Lie algebra $\operatorname{Lie}(G)$.
\end{thm}

For the class of nilpotent Leibniz algebras, the integration proposed by S.~Covez is global. 

\begin{thm}\label{thmnil} \cite{LM2022}
	Every nilpotent  real Leibniz algebra $\mg$ has a global integration into a Lie rack.
\end{thm}

When $\mg$ is a two-step nilpotent algebra, a Lie rack integrating $\mg$ can be defined without integrating the Leibniz algebras $2-$cocycle associated with $\mg$, since $[\mg,\mg] \subseteq \operatorname{Z}_l(\mg)$.

\begin{thm} \label{thm2} \cite{LM2022}
	Let $\mg$ be a two-step nilpotent algebra, let $\mg_0=\mg / \left[\mg,\mg\right]$ and let $\omega \colon \mg_0 \times \mg_0 \rightarrow \left[\mg,\mg\right]$ be the Leibniz algebras $2-$cocycle associated with the short exact sequence
	\begin{equation*}
		\begin{tikzcd}
			0\ar[r]
			& \left[\mg,\mg\right] \arrow[r, hookrightarrow]
			&\mg \arrow[r, twoheadrightarrow] &
			\mg_0 \ar[r]
			&0.
		\end{tikzcd}
	\end{equation*}
	Then the multiplication
	$$
	\left(x,a\right)\rhd\left(y,b\right)=(y,b)+[(x,a),(y,b)]=\left(y,b+\omega(x,y)\right)
	$$
	defines a Lie rack structure on the vector space $\mg_0\times \left[\mg,\mg\right]$, such that $(0,0)$ is the unit element and $\operatorname{T}_{(0,0)}(\mg_0\times \left[\mg,\mg\right], \rhd)$ is a Leibniz algebra isomorphic to $\mg$.
\end{thm}

Now we recall how to integrate the indecomposable nilpotent real Leibniz algebras with one-dimensional commutator ideal listed in the previous section. We suppose that $A \in \operatorname{M}_n(\mathbb{R})$ is the companion matrix of the power of an irreducible monic polynomial $f(x) \in \mathbb{R}[x]$. Thus $A$ is similar to the $n \times n$ Jordan block of eigenvalue $a \in \mathbb{R}$ or $A$ is similar to $J_R$, where $R \in \operatorname{M}_2(\mathbb{R})$ is the realification of some complex number $z=\alpha + i \beta$.

\begin{ex}\label{exHeis}
	Let $\mg=\mathfrak{l}_{2n+1}^A$ and let $\left\{e_1,\ldots,e_n,f_1,\ldots,f_n,h\right\}$ be a basis of $\mg$. The Lie global rack integrating $\mg$ is $R_{2n+1}^{A}=(\mg_0\times \mathbb{R}h,\rhd)$ with multiplication 
	\begin{gather*}
		(x_1,\ldots,x_n,y_1,\ldots,y_n,z)\rhd (x_1',\ldots,x_n',y_1',\ldots,y_n',z')=\\ =\left(x_1',\ldots,x_n',y_1',\ldots,y_n',z'+\sum_{i,j=1}^{n}\left[\left(\delta_{ij}+a_{ij}\right)x_iy_j'+\left(-\delta_{ij}+a_{ij}\right)x_i'y_j\right]\right).
	\end{gather*}
\end{ex}

\begin{definition}\cite{LM2022}
	$(R_{2n+1}^{A},\rhd)$ is called the \emph{Heisenberg rack}.
\end{definition}

When $A$ is the zero matrix, i.e.\	 when $\mg=\mh_{2n+1}$ is the $(2n+1)-$dimensional Heisenberg Lie algebra, we have that $R_{2n+1}^{0_{n\times n}}=\operatorname{Conj}(H_{2n+1})$.

In the same way, we can find the Lie racks integrating the \emph{Kronecker} algebra and the \emph{Dieudonné} algebra.

	\begin{definition}\cite{LM2022}
	Let $\mg=\mathfrak{k}_{n}$ and let $\left\{e_1,\ldots,e_n,f_1,\ldots,f_n,h\right\}$ be a basis of $\mg$. The \emph{Kronecker} rack $K_n=(\mg_0\times \mathbb{R}h,\rhd)$ is the Lie rack with multiplication 
	\begin{align*}
		x \rhd y = y + [x,y]_{\mathfrak{k}_n}, \quad \forall x,y \in K_n.
	\end{align*}
\end{definition}

\begin{definition}\cite{LM2022}
	Let $\mg=\mathfrak{d}_{n}$ and let $\left\{e_1,\ldots,e_{2n+1},h\right\}$ be a basis of $\mg$. The \emph{Dieudonn\'e} rack $D_n=(\mg_0\times \mathbb{R}h,\rhd)$ is the Lie rack with multiplication
	\begin{align*}
		x \rhd y = y + [x,y]_{\mathfrak{d}_n}, \quad \forall x,y \in D_n.
	\end{align*}
\end{definition}

One can check that $\operatorname{T}_{(0,0)} K_n \cong \mathfrak{k}_n$ and $\operatorname{T}_{(0,0)} D_n \cong \mathfrak{d}_n$.

\section{Isotopisms}

The concept of isotopism between two algebraic structures was introduced in~\cite{Albert} by A.\ A.\ Albert (see also \cite{Albert2} and \cite{IsotopismStory}) in order to classify non-associative algebras by generalizing the notion of isomorphism. We define here isotopisms for Leibniz algebras and racks.
	\begin{definition}\label{isotopismLeib}
		Let $(\mg,[-,-])$ and $(\mathfrak{h},[-,-]_\mh)$ be left Leibniz algebras over $\F$. An \emph{isotopism} between $\mg$ and $\mathfrak{h}$ is a triple of linear isomorphisms
		$$
		(f,g,h)\colon \mg \leftrightarrows \mathfrak{h}
		$$
		such that
		$$
		[f(x),g(y)]_{\mathfrak{h}}=h([x,y]), \quad \forall x,y \in \mg.
		$$
		We say that $\mg$ and $\mh$ are \emph{isotopic} Leibniz algebras.
	\end{definition}

An isotopism is said to be a \emph{right isotopism} if $f=h$ and a \emph{left isotopism} if $g=h$. We observe that, if $f \colon \mg \leftrightarrows \mathfrak{h}$ is an isomorphism of Leibniz algebras, then the triple $(f,f,f)$ is a (left and right) isotopism.

	\begin{definition}
		Let $(X,\rhd,1)$ and $(Y,\rhd_Y,1_Y)$ be left pointed racks. An \emph{isotopism} between $X$ and $Y$ is a triple of bijections
		$$
		(f,g,h) \colon X \leftrightarrows Y
		$$
		such that $f(1)=g(1)=h(1)=1_Y$ and
		$$
		f(x) \rhd_Y g(y)=h(x \rhd y), \quad \forall x,y \in X.
		$$
	\end{definition}

Also in this case, every isomorphism of racks can be seen as a left and right isotopism. Moreover, one can define the notion of \emph{Lie rack isotopism} by asking that $f,g,h$ are diffeomorphisms.

The following is the generalization of a result proved by A.\ A.\ Albert in \cite{Albert2} for isotopisms of \emph{quasigroups}.

\begin{rem}
	Every isotopism $(f,g,h)$ between two algebraic structures is isomorphic to a \emph{principal isotopism}, i.e.\ an isotopism of the form $(\tilde{f},\tilde{g},\operatorname{id})$.  To be more precise, let $(A,\cdot)$ and $(B,\circ)$ be two algebraic structures and let $(f,g,h)\colon A \leftrightarrows B$ be an isotopism. We define a binary operation $*\colon B \times B \rightarrow B$ by
	$$
	h(x) * h(y) = h(xy), \quad \forall x,y \in A.
	$$
	In this way $(A,\cdot)$ and $(B,*)$ are isomorphic and $h(x)*h(y)=f(x) \circ g(y)$. Thus $(hf^{-1})(\bar{x})*(hg^{-1})(\bar{y})=\bar{x} \circ \bar{y}$, where $\bar{x}=f(x)$, $\bar{y}=g(y)$, and we obtain the \textbf{principal isotopism} $(fh^{-1},gh^{-1},\operatorname{id}_B)\colon (B,\circ) \leftrightarrows (B,*)$.
	\[
	\begin{tikzcd}[row sep=large, column sep= 2cm]
		(B,\circ) \arrow{r}{(fh^{-1},gh^{-1},\operatorname{id}_B)}  & (B,*)\\
		(A,\cdot) \arrow{u}{(f,g,h)}  \arrow[dashed, swap]{ru}{h}&
	\end{tikzcd}
	\]
\end{rem}

In the case that $(A,\cdot)$ and $(B,\circ)$ are finite dimensional non-associative algebras over a field $\F$, such as Leibniz algebras, we can identify the underlying vector space of $A$ and $B$ as $\F^n$, where $n=\dim_\F A=\dim_\F B$, and $\operatorname{id}_B=\operatorname{id}_{\F^n}$.

From Definition \ref{isotopismLeib}, it turns out that a difference between isomorphisms and isotopisms is that one can find a Lie algebra that is isotopic to a Leibniz non-Lie algebra, as the following example shows.

	\begin{ex}
		Let $\mathfrak{h}_3$ be the three-dimensional Heisenberg Lie algebra, let $A=(a_{ij})_{i,j} \in \operatorname{GL}_2(\F)$, $B=(b_{ij})_{i,j}=A \cdot \operatorname{diag}\lbrace \lambda, \mu \rbrace$, with $\lambda,\mu \in \F^*$ and let $\mg_{\lambda,\mu}^{A}$ be the nilpotent Leibniz algebra with one-dimensional commutator ideal with structure matrix
		$$
		\begin{pmatrix}
			0 & a_{11} b_{22} - a_{21}b_{12} & 0 \\
			a_{12}b_{21} - a_{22}b_{11} & 0 & 0 \\
			0 & 0 & 0 \\
		\end{pmatrix}.
		$$
		Then the triple of linear isomorphisms $(f,g,\operatorname{id}_{\mathbb{F}^3})\colon \mg_{\lambda,\mu}^{A} \leftrightarrows \mathfrak{h}_3$, where $f$ and $g$ are defined respectively by the matrices
		\[
		\begin{pmatrix}
			a_{11} & a_{21} & 0 \\
			a_{12} & a_{22} & 0 \\
			0 & 0 & \alpha \\
		\end{pmatrix}\text{ and }
	\begin{pmatrix}
		b_{11} & b_{21} & 0 \\
		b_{12} & b_{22} & 0 \\
		0 & 0 & \beta \\
	\end{pmatrix}, 
\text{ with } \alpha,\beta\in\mathbb{F}^\ast
	\]
		is a principal isotopism. Indeed, for every $x\equiv(x_1,x_2,x_3), y\equiv(y_1,y_2,y_3)\in \F^3$, we have
		\begin{align*}
			\left[f(x), g(y)\right]_{\mathfrak{h}_3}=& (0,0, (a_{11}x_1+a_{12}x_2)(b_{21}y_1+b_{22}y_2)-(a_{21}x_1+a_{22}x_2)(b_{11}y_1+b_{12}y_2))\\
			=& (0,0,(a_{11}b_{22}-a_{21}b_{12})x_1y_2-(a_{12}b_{21}-a_{22}b_{11})x_2y_1)=\left[x,y\right]_{\mg_{\lambda,\mu}^{A}},
		\end{align*}
		since $a_{11}b_{21}-b_{11}a_{21}=a_{12}b_{22}-b_{12}a_{22}=0$.
	\end{ex}
	
		\noindent One can check that, if $\lambda=\dfrac{2}{\det(A)}-\mu$, then $\mg_{\lambda,\mu}^{A}=\mathfrak{l}_3^{\alpha}$, where $\alpha=\mu\det(A)-1$. Indeed $\mg_{\lambda,\mu}^{A}$ and $\mathfrak{l}_3^{\alpha}$ have the same structure matrix, that is
		\[
		\begin{pmatrix}
			0 & \mu\det(A) & 0\\
			-\lambda\det(A) & 0 & 0\\
			0 & 0 & 0
		\end{pmatrix}=
		\begin{pmatrix}
				0 & 1+\alpha & 0\\
		-1+\alpha & 0 & 0\\
			0 & 0 & 0
		\end{pmatrix}.\] 
	     Conversely, for every $\alpha \in \F^{\ast} \setminus \lbrace \pm 1 \rbrace$, if we choose $A=\operatorname{I}_2$, $\lambda=1-\alpha$ and $\mu=1+\alpha$, then $\mg_{\lambda,\mu}^{\operatorname{I}_2}=\mathfrak{l}_3^{\alpha}$. Thus $\mathfrak{h}_{3}$ and $\mathfrak{l}_3^{\alpha}$ are isotopic for every $\alpha \in \F$, $\alpha \neq \pm 1$. We note that this result is not true if $\alpha=\pm 1$, in fact in the next section we will show that $\mathfrak{l}_3^{\pm 1}$ and $\mathfrak{h}_3$ are not isotopic.
	     
	     In a similar way, if $\mu=-\lambda=\dfrac{1}{\operatorname{det}(A)}$, then the structure matrix of $\mg_{\lambda,\mu}^A$ is 
		$$
		\begin{pmatrix}
			0 & 1 & 0 \\
			1 & 0 & 0 \\
			0 & 0 & 0 \\
		\end{pmatrix},
		$$
		i.e.\ $\mg_{\lambda,\mu}^A=\mathfrak{k}_1$ is the Kronecker algebra. Thus $\mathfrak{l}_3^{\alpha}$ is isotopic to $\mathfrak{k}_1$, for every $\alpha \neq \pm 1$.
		
\section{Isotopisms of two-step nilpotent algebras and Lie racks}\label{lastsection}
	
		The main goal of this section is to investigate the isotopism classes of indecomposable nilpotent Leibniz algebras with one-dimensional commutator ideal over a field $\F$, with $\operatorname{char}(\F) \neq 2$. Moreover we will show that, when $\F=\R$, two Leibniz algebras of this type are isotopic if and only if the Lie racks integrating them are isotopic. For doing this, we need to find new isotopism invariants for Leibniz algebras and Lie racks.
		
		We start with the following definition, which was given by M.\ Elhamdadi and E.\ M.\ Moutuou in \cite{StabRack}.
		
		\begin{definition}
			Let $X$ be a left rack. The \emph{center} of $X$ is
			$$
			\operatorname{Z}(X)=\lbrace x \in X \; | \; x \rhd y = y, \; \forall y \in X \rbrace.
			$$
		\end{definition}

\noindent We observe that $\operatorname{Z}(X)$ is a subrack of $X$ since $1 \rhd y=y$ and
$$
(x \rhd x') \rhd y = x' \rhd y = y
$$
for every $x,x' \in \operatorname{Z}(X)$ and for every $y \in X$. In \cite{StabRack} the center of $X$ is called the \emph{stabilizer} of $X$ and it is denoted by $\operatorname{Stab}(X)$. We refer to it as the center for the following reasons:
\begin{enumerate}
 \item[(1)] if $G$ is a group, then $\operatorname{Z}(\operatorname{Conj}(G))=\operatorname{Z}(G)$;
 \item[(2)] if $X$ is a Lie rack and $\mg = \operatorname{T}_1 X$, then $\operatorname{T}_1 \operatorname{Z}(X) = \operatorname{Z}_l(\mg)$.
\end{enumerate}
	
	\begin{prop}{\ } \label{invariants}
		\begin{itemize}
			\item[(i)] Let $\mg,\mathfrak{h}$ be Leibniz algebras and let $(f,g,h)\colon \mg \leftrightarrows \mathfrak{h}$ be an isotopism. Then
			$$
			f(\operatorname{Z}_l(\mg)) = \operatorname{Z}_l(\mathfrak{h}), \quad g(\operatorname{Z}_r(\mg)) = \operatorname{Z}_r(\mathfrak{h}), \quad h([\mg,\mg])=[\mh,\mh]
			$$
			and the dimensions of the left center, of the right center and of the commutator ideal are \emph{isotopism invariants}.
			\item[(ii)] Let $X,Y$ be left racks and let $(f,g,g)\colon X \leftrightarrows Y$ be a \emph{left isotopism}. Then
			$$
			f(\operatorname{Z}(X)) = \operatorname{Z}(Y).
			$$
			In the case of Lie racks, the dimension of the center is a \emph{left isotopism inviariant}.
		
		\end{itemize}
		
	\end{prop}

\begin{proof}{\ }
	\begin{itemize}
		\item[(i)] Let $x,y \in \mh$ such that $x=f(\bar{x})$, $y=g(\bar{y})$, with $\bar{x},\bar{y} \in \mg$. If $\bar{x} \in \operatorname{Z}_l(\mg)$, then
		$$
		[x,y]_{\mh}=[f(\bar{x}),g(\bar{y})]_{\mh}=h([\bar{x},\bar{y}])=h(0)=0
		$$
		and $x \in \operatorname{Z}_l(\mh)$. Conversely, if $x \in \operatorname{Z}_l(\mh$), then
		$$
		0=[x,y]_\mh=[f(\bar{x}),g(\bar{y})]_\mh=h([x,y]).
		$$
		Thus $[\bar{x},\bar{y}] \in \operatorname{Ker}(h)=0$, i.e.\ $\bar{x} \in \operatorname{Z}_l(\mg)$ and $f(\operatorname{Z}_l(\mg)) = \operatorname{Z}_l(\mathfrak{h})$.
		In a similar way, one can check that $g(\operatorname{Z}_r(\mg)) = \operatorname{Z}_r(\mathfrak{h})$. Finally, for every $\bar{x},\bar{y} \in \mg$
		$$
		h([\bar{x},\bar{y}])=[f(\bar{x}),g(\bar{y})]_\mh \in [\mh,\mh].
		$$
		Coversely, for any $x,y \in \mh$, with $x=f(\bar{x})$ and $y=g(\bar{y})$, we have
		$$
		[x,y]_{\mh}=h([\bar{x},\bar{y}]) \in h([\mg,\mg])
		$$
		and $h([\mg,\mg])=[\mh,\mh]$.
		\item[(ii)] Let $x,y \in Y$ such that $x=f(\bar{x})$, $y=g(\bar{y})$, with $\bar{x},\bar{y} \in X$. If $\bar{x} \in \operatorname{Z}(X)$, then
		$$
		x \rhd y = f(\bar{x}) \rhd_Y g(\bar{y}) = g(\bar{x} \rhd \bar{y}) = g(\bar{y})=y
		$$
		and $x \in \operatorname{Z}(Y)$. Conversely, if $ x \in \operatorname{Z}(Y)$, then
		$$
		g(\bar{y})=y=x \rhd_Y y = f(\bar{x}) \rhd_Y g(\bar{y}) = g(\bar{x} \rhd \bar{y})
		$$
		and the statement is proved, since $g$ is injective.
	\end{itemize}
\end{proof}

The last proposition allows us to conclude that the algebras $\mathfrak{l}_3^{\pm 1}$ and $\mathfrak{h}_3$ are not isotopic, since $\dim_\F \operatorname{Z}_l(\mathfrak{l}_3^{1})=2$ (see Section \ref{HeisC}) and $\dim_\F \operatorname{Z}(\mathfrak{h}_3)=1$. Moreover, we recall that $\mathfrak{l}_3^{1}$ and $\mathfrak{l}_3^{-1}$ are isomorphic, and $\mathfrak{h}_3$ is isotopic to $\mathfrak{l}_3^{\alpha}$, for every $\alpha \neq \pm 1$, thus $\mathfrak{l}_3^{\alpha}$ and $\mathfrak{l}_3^{\pm 1}$ are not isotopic. This remark, combined with the one we illustrated at the end of the previous section, leeds to the following.

\begin{prop}\label{propisotop}
	Let $\mg$ be a three-dimensional indecomposable nilpotent Leibniz algebra with one-dimensional commutator ideal over a field $\F$, with $\operatorname{char}(\F) \neq 2$. Then either $\mg$ is isomorphic to the Heisenberg algebra $\mathfrak{l}_3^{1}$, or $\mg$ is isotopic to the Heisenberg Lie algebra $\mathfrak{h}_3$.
\end{prop}

We want now to extend this result to the case that $\dim_\F \mg > 3$.

	\begin{thm}\label{thmiso}
		Let $\F$ be a field with $\operatorname{char}(\F) \neq 2$ and let $\mg$ be a nilpotent Leibniz algebra with $\dim_\F \mg=t \geq 3$ and $\dim_\F [\mg,\mg]=1$.
		\begin{itemize}
			\item[(i)] If $t=2n+2$, then $\mg$ is isomorphic to the \emph{Dieudonn\'e} algebra $\mathfrak{d}_n$.
			\item[(ii)] If $t=2n+1$, then either $\mg$ is isomorphic to the Heisenberg algebra $\mathfrak{l}_{2n+1}^{J_1}$, where $J_1$ is the $n \times n$ Jordan block of eigenvalue $1$, or $\mg$ is isotopic to the Heisenberg Lie algebra $\mathfrak{h}_{2n+1}$.
			\item[(iii)] The Heisenberg algebras $\mathfrak{l}_{2n+1}^{J_1}$ and $\mathfrak{l}_{2n+1}^{J_{-1}}$ are not isotopic to any Lie algebra.
		\end{itemize}
	\end{thm}

\begin{proof}{\ }
	\begin{itemize}
		\item[(i)] If $\dim_\F \mg$ is even, then $\mg$ must be isomorphic to the Dieudonn\'e algebra $\mathfrak{d}_n$, where $t=2n+2$.
		\item[(ii)] If $\dim_\F \mg = 2n+1$ and $\mg$ is not isomorphic to $\mathfrak{l}_{2n+1}^{J_{1}}$, then $\mg \cong \mathfrak{k}_n$ or $\mg \cong \mathfrak{l}_{2n+1}^{A}$, where $A$ is the companion matrix of the power of an irreducible monic polynomial $p(x)^k \in \F[x]$ and it is not similar to the Jordan blocks $J_1$ and $J_{-1}$. We want to show that both these Leibniz algebras are isotopic to the Heisenberg Lie algebra $\mathfrak{h}_{2n+1}$. The Leibniz algebra $\mathfrak{k}_n$ is isotopic to the $(2n+1)$-dimensional Heisenberg Lie algebra via the left principal isotopism $(f,\operatorname{id}_{\mathbb{F}^t},\operatorname{id}_{\mathbb{F}^t})$, where
		\begin{gather*}
			f(x_1,\ldots,x_n,y_1,\ldots,y_n,z)=\\=(x_1+x_2,x_2+x_3,\ldots,x_{n-1}+x_n,x_n,-y_1,y_1-y_2,y_2-y_3,\ldots,y_{n-1}-y_n,z),
		\end{gather*}
		for every $(x_1,\ldots,x_n,y_1,\ldots,y_n,z)\in\mathbb{F}^t$. The Heisenberg algebra $\mathfrak{l}_{2n+1}^{A}$ is isotopic to $\mathfrak{h}_{2n+1}$ via the triple $(f, \operatorname{id}_{\mathbb{F}^t},\operatorname{id}_{\mathbb{F}^t})$, where $f$ is the linear isomorphism defined by the matrix 
		\[\begin{pmatrix}
			 & & & \vline & & & & \vline & 0\\
			& I_n+A & & \vline & & \bigzero & & \vline & \vdots\\
			& & & \vline & & & & \vline & 0\\
			\hline
			 & & & \vline & & & & \vline & 0\\
			 & \bigzero & & \vline & & I_n-A^T& & \vline & \vdots\\
			  & & & \vline & & & & \vline & 0\\
			  \hline
			0 & \cdots \cdots & 0 & \vline & 0 & \cdots \cdots & 0 & \vline & 1 
		\end{pmatrix}\]
 We observe that $f$ is a bijection since the matrices $I_n+A$ and $I_n-A^T$ are invertible. Indeed one has
	$$\operatorname{det}(xI_n+A)=(-1)^n f(x),$$
	$$\operatorname{det}(xI_n-A^T)=\operatorname{det}(xI_n-A)=f(x)$$
	and for $x=1$
	$$
	\operatorname{det}(I_n+A)=(-1)^n f(-1),
	$$
	$$
	\operatorname{det}(I_n-A^T)=f(1).
	$$
	Thus $f$ is not invertible if and only if $p(1)=0$ or $p(-1)=0$. Since $p(x)$ is irreducible over $\F$, this would imply that $p(x)=x-1$ or $p(x)=x+1$ and $A$ would be similar to the $n \times n$ Jordan block $J_1$ or $J_{-1}$. We have a contradiction since we supposed that $\mg$ is not isomorphic to the Heisenberg algebra $\mathfrak{l}_{2n+1}^{J_{1}}$. Finally, by Proposition \ref{invariants} we have that $\mathfrak{l}_{2n+1}^{J_{1}}$ and $\mathfrak{l}_{2n+1}^{J_{-1}}$ are not isotopic to $\mathfrak{h}_{2n+1}$ since
		$$
		\dim_\F \operatorname{Z}_l(\mathfrak{l}_{2n+1}^{J_{\pm 1}})=2,
		$$
		$$
		\dim_\F \operatorname{Z}(\mathfrak{h}_{2n+1})=1
		$$
		and the dimension of the left center is an isotopism invariant.
		\item[(iii)] If we suppose that there exists a Lie algebra $\mathfrak{p}$ and an isotopism
		$$
		(f,g,h)\colon \mathfrak{l}_{2n+1}^{J_{1}} \leftrightarrows \mpp,
		$$
		then, by Proposition \ref{invariants}, $\dim_\F[\mpp,\mpp]=1$ and $\dim_\F \operatorname{Z}(\mpp)=2$. For the classification of Lie algebras with one-dimensional commutator ideal (see Section 3 of \cite{erdmann2006introduction}), $\mpp=\mpp_1 \oplus \mpp_2$, where $\mpp_1$ is an abelian algebra and $\mpp_2$ is either the non-abelian Lie algebra of dimension $2$ (i.e.\ $\mpp_2$ has basis $\lbrace e_1,e_2 \rbrace$ and the Lie bracket is defined by $[e_1,e_2]=e_1$), or it is an Heisenberg algebra $\mathfrak{h}_{2k+1}$. In the first case, $\dim_\F \mpp_1$ must be odd. Thus $\dim_\F\operatorname{Z}(\mpp)$ is also odd, since $\operatorname{Z}(\mpp)=\mpp_1$. In the second one, $\dim_\F \mpp_1$ must be even and $$\dim_\F \operatorname{Z}(\mpp)=\dim_\F \mpp_1 + \dim_\F \operatorname{Z}(\mh_{2k+1})=\dim_\F \mpp_1 +1$$ is odd. Hence, in both cases we have a contradiction and $\mathfrak{l}_{2n+1}^{J_{1}}$ cannot be isotopic to a Lie algebra. The statement is also valid for the Leibniz algebra $\mathfrak{l}_{2n+1}^{J_{-1}}$, since it is isomorpich to $\mathfrak{l}_{2n+1}^{J_{1}}$.
	\end{itemize}
\end{proof}

In the last theorem we showed that an isotopism between two nilpotent Leibniz algebras with one-dimensional commutator ideal $\mg,\mh$ can be always chosen of the form $(f,\operatorname{id}_{\F^t},\operatorname{id}_{\F^t})$. In the next lemma we show that, if $\F$ is the field of real numbers, such an isotopism turns out to be also an isotopism between the Lie racks integrating $\mg$ and $\mh$. We recall that, by Theorem \ref{thm2}, we can identify the underlying vector space of a Lie rack $X$ integrating a two-step nilpotent Leibniz algebra, with its tangent space $\operatorname{T}_1 X$.
	
	\begin{lemma}\label{lemmaracks}
		Let $\F=\mathbb{R}$ and let $\mg,\mathfrak{h}$ be two-step nilpotent algebras with $\dim_\R \mg=\dim_\R \mh=t$. Let $X,Y$ be the Lie racks integrating $\mg$ and $\mh$ respectively. 
		\begin{itemize}
			\item[(i)] If $(f,g,g) \colon \mg \leftrightarrows \mh$ is a left Leibniz algebra isotopism, then it is also a Lie rack isotopism between $X$ and $Y$.
			\item[(ii)] If $(f,g,h) \colon X \leftrightarrows Y$ is a Lie rack isotopism, then $g=h$ and it turns out to be a Leibniz algebra isotopism between $\mg$ and $\mh$.
			\item[(iii)] Let $\dim_\R [\mg,\mg] = \dim_\R [\mh,\mh]=1$. Then $\mg$ is isotopic to $\mh$ if and only if $X$ is isotopic to $Y$.
		\end{itemize} 
	\end{lemma}
	
\begin{proof}
	The multiplications of the Lie racks $X$ and $Y$ can be written as
	$$
	x \rhd y = y + [x,y]_{\mg}, \quad x' \rhd_Y y' = y' + [x',y']_{\mh}, \quad \forall x,y \in X,\; \forall x',y' \in Y.
	$$
	\begin{itemize}
		\item[(i)] If $(f,g,g)\colon \mg \leftrightarrows \mathfrak{h}$ is an isotopism of Leibniz algebra, then
		$$
		f(x) \rhd_{Y} g(y)= g(y) + [f(x),g(y)]_{\mathfrak{h}}=g(y)+g([x,y])=g(x \rhd y),
		$$
		for every $x,y \in X$.
		\item[(ii)] If $(f,g,h) \colon X \leftrightarrows Y$ is an isotopism of Lie racks, then
		$$
		g(y)=1_Y \rhd_Y g(y)=f(1_X) \rhd_Y g(y) = h(1_X \rhd y) =h(y), \; \; \forall y \in X,
		$$
		thus $g=h$ and the triple $(f,g,g)$ becomes an isotopism between $\mg$ and $\mh$, since
		$$
		g(y) + [f(x),g(y)]_{\mh}=f(x) \rhd_Y g(y) = g(x \rhd y) = g(y) + g([x,y])
		$$
		for every $x,y \in X$, and then
		$$
		[f(x),g(y)]_{\mh}=g([x,y]).
		$$
		\item[(iii)] If $X$ and $Y$ are isotopic Lie racks, then from (ii) $\mg$ and $\mh$ are isotopic Leibniz algebra. Conversely, if $\mg$ and $\mh$ are isotopic nilpotent Leibniz algebras with one-dimensional commutator ideal, then by Theorem \ref{thmiso}, an isotopism between them can be chosen of the form $(f,\operatorname{id}_{\F^t},\operatorname{id}_{\F^t})$ and, from (ii), it becomes an isotopism between $X$ and $Y$.
	\end{itemize}
	\end{proof} 

The last result allows us to describe the isotopism classes of Lie global racks integrating the indecomposable nilpotent Leibniz algebras with one-dimensional commutator ideal.

	\begin{thm}
		Let $\F=\mathbb{R}$ and let $X$ be a Lie rack integrating a nilpotent Leibniz algebra $\mg$ with $\dim_\F \mg=t$ and $\dim_\F [\mg,\mg]=1$.
		\begin{itemize}
			\item[(i)] If $t=2n+2$, then $X$ is isomorphic to the \emph{Dieudonn\'e} rack $D_n$.
			\item[(ii)] If $t=2n+1$, then either $X$ is isomorphic to the Heisenberg rack $R_{2n+1}^{J_1}$, where $J_1$ is the $n \times n$ Jordan block of eigenvalue $1$, or $X$ is isotopic to the conjugation of the Heisenberg Lie group $\operatorname{Conj}(H_{2n+1})$.
			\item[(iii)] The Heisenberg racks $R_{2n+1}^{J_1}$ and $R_{2n+1}^{J_{-1}}$ are not isotopic to any Lie quandle. 
		\end{itemize}
	\end{thm}

\begin{proof}{\ }
	\begin{itemize}
		\item[(i)] If $t$ is even, then $\mg \cong \mathfrak{d}_n$ and $X$ is isomorphic to the Dieudonné rack $D_n$.
			\item[(ii)] If $t=2n+1$ and $\mg \cong \mathfrak{l}_{2n+1}^{J_{1}}$, then $X$ is isomorphic to the Heisenberg rack $R_{2n+1}^{J_{1}}$. If this is not the case, then $\mg$ is isotopic to $\mh_{2n+1}$ and $X$ is isomorphic either to the Kronecker rack $K_n$, or to $R_{2n+1}^{A}$, where $A=J_{\alpha}$ with $\alpha \in \mathbb{R} \setminus \lbrace \pm 1 \rbrace$, or $A=J_R$ and $R=R_{\alpha,\beta}$ as in Section \ref{secHeisReal}. Note that the last case is admissible if and only if $n$ in even. Since both $\mathfrak{k}_n$ and $\mathfrak{l}_{2n+1}^{A}$ are isotopic to the Heisenberg Lie algebra $\mathfrak{h}_{2n+1}$ with a left principal isotopism $(f,\operatorname{id}_{\F^{t}},\operatorname{id}_{\F^{t}})$, by Lemma \ref{lemmaracks} we can conclude that both $K_n$ and $R_{2n+1}^{A}$ are isotopic to $\operatorname{Conj}(H_{2n+1})$. Thus $X$ is isotopic to the conjugation of the Heisenberg Lie group. Finally $R_{2n+1}^{J_{1}}$ and $\operatorname{Conj}(H_{2n+1})$ are not isotopic. Indeed, if we suppose that there exists an isotopism
		$$
		(f,g,h) \colon R_{2n+1}^{J_{1}} \leftrightarrows \operatorname{Conj}(H_{2n+1}),
		$$
		then, by Lemma \ref{lemmaracks}, $g=h$ and we obtain a contradiction, since $(f,g,g)$ would become an isotopism between the Leibniz algebras $\mathfrak{l}_{2n+1}^{J_{ì1}}$ and $\mathfrak{h}_{2n+1}$. Another way to show that $R_{2n+1}^{J_{1}}$ and $\operatorname{Conj}(H_{2n+1})$ are not isotopic racks is to see that
		$$
		\operatorname{Z}(R_{2n+1}^{J_{1}})=\lbrace (0,y,z) \; | \; y,z \in \R \rbrace,
		$$
		$$
		\operatorname{Z}(\operatorname{Conj}(H_{2n+1}))=\lbrace (0,0,z) \; | \; z \in \R \rbrace,
		$$
		where we use the same notation of Example \ref{exHeis}, and we proved in Proposition \ref{invariants} that the dimension of the center of a Lie rack is a left isotopism invariant.
		\item[(iii)] If $R_{2n+1}^{J_1}$ were isotopic to a Lie quandle $X$, then $\mathfrak{l}_{2n+1}^{J_1}$ would be isotopic to the Leibniz algebra $\mg=\operatorname{T}_1 X$. By Theorem \ref{thmquandle}, $\mg$ is a Lie algebra and $X=\operatorname{Conj}(G)$, where $\operatorname{Lie}(G)=\mg$. Thus $\mathfrak{l}_{2n+1}^{J_1}$ would be isotopic to a Lie algebra and, by (iii) of Theorem \ref{thmiso}, this is a contradiction. The same argument can be used for the Lie rack $R_{2n+1}^{J_{-1}}$.
	\end{itemize}
\end{proof}

In the last section we saw that, though the problem of finding isomorphism classes of the complex Heisenberg algebras $\mathfrak{l}_{2n+1}^{J_a}$ is still open, the notion of isotopism allow us to classify all the possible non-isotopic nilpotent Leibniz algebras with one-dimensional commutator ideal. Moreover, this induces also a classification of the non-isotopic Lie racks integrating such class of nilpotent Leibniz algebras.

%

\printbibliography
	
\end{document}